\documentclass[reqno]{amsart}
\usepackage{amssymb,amsmath,amsfonts,amsthm,enumitem}
\usepackage[numbers,sort&compress]{natbib}
\usepackage{hyperref}
\hypersetup{colorlinks=true,linkcolor=green,pdfborderstyle={/S/U/W 1}}

\usepackage{tikz}
\tikzset{
  vert/.style={circle, draw=black!100,fill=black!100,thick, inner sep=0pt, minimum size=2mm}, 
  smallvert/.style={circle, draw=black!100,fill=black!100,thick, inner sep=0pt, minimum size=1mm},
  empty/.style={draw=none, fill=none, minimum size=0mm, inner sep=0pt}
}

\theoremstyle{plain}
\newtheorem{theorem}  {Theorem}  [section]
\newtheorem{lemma}  [theorem]   {Lemma}
\newtheorem{corollary}[theorem] {Corollary}
\newtheorem{fact}[theorem]{Fact}{\bf}{\it}

\theoremstyle{definition}
\newtheorem{remark}[theorem] {Remark}
\newtheorem{definition}[theorem] {Definition}

\let\nadj\nsim%

\let\polishlcross=\l
\def\l{\ifmmode\ell\else\polishlcross\fi}

\makeatletter
\newcommand{\mylabel}[2]{#2\def\@currentlabel{#2}\label{#1}}
\makeatother

\usepackage{relsize}

\newcommand{\cH}{\mathcal{H}}
\newcommand{\cU}{\mathcal{U}}

 \DeclareMathOperator{\bHom}{Hom}
 \DeclareMathOperator{\recol}{Recol}
 \DeclareMathOperator{\recon}{Recon}
 \DeclareMathOperator{\Col}{Col}
 \DeclareMathOperator{\CSP}{CSP}
 \DeclareMathOperator{\PSPACE}{PSPACE}
 \DeclareMathOperator{\NP}{NP}

\newcommand\ec[1]{[#1]}    
\newcommand{\cl}[1]{[\![#1]\!]}  
\newcommand{\con}{\cdot}
\newcommand{\rev}[1]{{#1^{-1}}}

\begin{document}

\author{Jae-baek Lee}
\address{J. Lee: University of Victoria, Canada}
\email{dlwoqor0923@uvic.ca}
\author{Jonathan A. Noel}
\address{J. Noel: University of Victoria, Canada}
\email{noelj@uvic.ca}
\author{Mark Siggers}
\address{M. Siggers: Kyungpook National University, Republic of Korea}
\email{mhsiggers@knu.ac.kr}
\thanks{The last author was supported by Korean NRF Basic Science Research Program (2018-R1D1A1A09083741) funded by the Korean government (MEST)
          and the Kyungpook National University Research Fund.}

\title{Recolouring Homomorphisms to triangle-free reflexive graphs}



\begin{abstract}
  For a graph $H$, the $H$-recolouring problem $\operatorname{Recol}(H)$ asks, for two given homomorphisms from a given  graph $G$ to $H$, if one can get between them by a sequence of homomorphisms of $G$ to $H$ in which consecutive homomorphisms differ on only one vertex.  We show that, if $G$ and $H$ are reflexive and $H$ is triangle-free, then this problem can be solved in polynomial time.   This shows, at the same time, that the closely related $H$-reconfiguration problem $\operatorname{Recon}(H)$ of deciding whether two given homomorphisms from a given graph $G$ to $H$ are in the same component of the Hom-graph $\operatorname{{Hom}}(G,H)$, can be solved in polynomial time for triangle-free reflexive graphs $H$. 
\end{abstract}
\maketitle

\section{Introduction}

Reconfiguration, in various settings, is the common notion of moving between solutions of a combinatorial problem via small changes; see van den Heuvel~\cite{vdH13Survey} and Nishimura~\cite{Nishimura} for detailed surveys. For example, a reconfiguration, or \emph{recolouring}, between two graph colourings is a sequence of colourings in which consecutive elements differ on one vertex.  
 The \emph{$k$-recolouring problem} is to decide, for two given $k$-colourings of $G$, whether or not there is a
 recolouring between them.  Cereceda, van den Heuvel and Johnson \cite{CvdHJ11} showed that $3$-recolouring is polynomial time solvable, while Bonsma and Cereceda \cite{BonCer09} showed that $k$-recolouring is $\PSPACE$-complete for $k \geq 4$.  One should compare this to the $k$-colouring problem which is well known to be polynomial time solvable for $k=2$ but $\NP$-complete for $k \geq 3$.  

 The $k$-colouring problem generalizes to the $H$-colouring problem $\Col(H)$ for a graph $H$ and to the constraint satisfaction problem $\CSP(\cH)$ for a relational structure $\cH$. The computational complexity of $\Col(H)$, and its myriad of generalizations and variations, have been well studied over the last 50 years, culminating in the recent CSP dichotomy theorem of Bulatov \cite{BulCSP} and Zhuk \cite{Zhuk}, which says that $\CSP(\cH)$ is either 
 polynomial time solvable or $\NP$-complete for every relational structure $\cH$.    

Analogously, the $k$-recolouring problem also generalizes to reconfiguration problems for $H$-colourings and CSP, both of which are well studied; see, e.g.,~\cite{Wroch15,BMMN16reconfig,LNS20,BLMNS20,BLS18} and~\cite{GKMP09,MTY11,MTY10,CDEHW20,MNPR17,BMNR14,Schwer12,HIZ18}, respectively. In particular, Gopalan et al.  \cite{GKMP09} proved a dichotomy theorem for the reconfiguration variation of $\CSP(\cH)$ for structures $\cH$ with two vertices. In this paper, we focus on the complexity of $H$-colouring reconfiguration in the setting of graphs which are \emph{reflexive} in the sense that they have a loop on every vertex; graphs with no loops at all are called \emph{irreflexive}. While such problems are also interesting for digraphs~\cite{BLS18}, in this paper we restrict our attention mainly to symmetric (i.e. undirected) graphs.

Recall that a {\em homomorphism} $\phi:G \to H$ from a graph $G$ to a graph $H$ is an edge-preserving vertex map. It is also known as an \emph{$H$-colouring} of $G$ due to the fact that a $K_k$-colouring of $G$, for the irreflexive clique $K_k$, is equivalent to a $k$-colouring. A \emph{recolouring} between two $H$-colourings of an irreflexive graph $G$ is a sequence of $H$-colourings in which consecutive elements differ on one vertex. If $G$ has loops, then a recolouring is defined in the same way, but with the added restriction that, in any given step, the image of a reflexive vertex can only move to a neighbour of its previous image. An instance $(G, \phi, \psi)$ of the {\em $H$-recolouring problem} $\recol(H)$ consists of a graph $G$ and two $H$-colourings of $G$; the task is to decided whether or not there is a recolouring between them.

The {\em Hom-graph} $\bHom(G,H)$ for two graphs $G$ and $H$ has the set of all homomorphisms from $G$ to $H$ as its vertex set, and two homomorphisms $\phi$ and $\psi$ are adjacent if, for all edges $uv$ in $G$, $\phi(u)\psi(v)$ is an edge in $H$. It has been observed by several authors (e.g.~\cite{Wroch15,BrewsterNoel}) that, for undirected graphs, the $H$-recolouring problem is essentially the same as asking whether there is a path between a given pair of homomorphisms in $\bHom(G,H)$; we call this the \emph{$H$-reconfiguration problem} and denote it by $\recon(H)$. In one direction, any path from $\phi$ to $\psi$ in $\bHom(G,H)$ can be converted into a path in which any two consecutive $H$-colourings differ on one vertex by simply changing the colours ``one at a time.'' In the other direction, if $(G,\phi,\psi)$ is a YES instance for $\recol(H)$, then the sequence of colourings is a path from $\phi$ to $\psi$ in $\bHom(G,H)$ unless there exists an isolated vertex $v$ of $G$ with a loop which is mapped to different components of $H$ by $\phi$ and $\psi$.\footnote{In contrast, it is interesting to remark that the problems $\recol(H)$ and $\recon(H)$ are not so closely related in the digraph case. For given $H$ they may have different sets of connected YES instances; though we do not have examples yet where they have different complexity~\cite{BLS18}.} While we will state our results in terms of $\recol(H)$ in the introductory sections, in the later sections it is convenient to speak in terms of paths in the graph $\bHom(G,H)$,  and when we do so, we talk of $\recon(H)$ and reconfigurations rather than $\recol(H)$ and recolourings.

The diameter of $\bHom(G,H)$ can be superpolynomial in $|V(G)|$ (see, e.g.,~\cite{BonCer09}) and so, a priori, a YES instance of $\recol(H)$ may not have a certificate of polynomial size. Thus, the natural complexity class for $\recol(H)$ is $\PSPACE$, rather than  $\NP$.

Some of the known results for irreflexive graphs are summarized as follows. The case that $H$ is a clique follows from the results of~\cite{BonCer09,CvdHJ11} for $k$-recolouring discussed earlier. In \cite{BMMN16reconfig}, the dichotomy for cliques was generalized to `circular cliques' $G_{p/q}$; specifically, $\recol(G_{p/q})$ is polynomial time solvable if $p/q < 4$ and $\PSPACE$-complete if $p/q \geq 4$. The polynomial time algorithm for $3$-recolouring in \cite{CvdHJ11} had a topological flavor which Wrochna distilled and extended in \cite{Wroch15} to prove the very general result that $\recol(H)$ is polynomial time solvable for any graph $H$ not containing a $4$-cycle. In~\cite{LNS20}, we proved that, if $H$ is a $K_{2,3}$-free quadrangulation of the sphere different from the $4$-cycle, then $\recol(H)$ is $\PSPACE$-complete. 
 
 In the current paper, we use several ideas of these earlier papers, and adapt Wrochna's language and many of his ideas, to prove an analogue of his result for reflexive graphs $G$ and $H$. Reflexive graphs are actually a smoother setting for some of these ideas; e.g. the fact that two consecutive vertices in a path in $G$ can map to the same vertex of $H$ eliminates a parity obstruction which turns up in~\cite{Wroch15}.  We prove the following. 
 
\begin{theorem}\label{thm:main}
  If $H$ is a reflexive triangle-free graph, then $\recol(H)$ can be solved in polynomial  time for reflexive instances.
\end{theorem}

 In contrast, we recently proved in \cite{LNS20} that $\recol(H)$ is $\PSPACE$-complete when $H$ is any $K_4$-free reflexive triangulation of the sphere, other than a triangle itself. We refer to this paper for a deeper suggestion of the connection of topology to the complexity of the reconfiguration problem. We remark that the argument of Wrochna~\cite{Wroch15} directly generalizes to get that $\recol(H)$ is solvable in polynomial time for general instances if $H$ is a graph which may have loops on some vertices which does not contain a $4$-cycle, a triangle with a loop on one vertex, or two adjacent vertices with loops; see the footnote on p.~330 of~\cite{Wroch20}.
 
 Before we move on to the proof of the main theorem, we observe that it implies the following result for a less general class of graphs $H$, but more general instances. Recall that the \emph{girth} of a graph is the length of its shortest cycle. 
 
\begin{corollary}
 If $H$ is a reflexive symmetric graph with girth at least $5$, then $\recol(H)$ can be solved in polynomial time for all instances. 
\end{corollary}

\begin{proof}
  Let $(G,\phi,\psi)$ be an instance of $\recol(H)$. If $G$ has isolated vertices, then we can simply change the colour of every such vertex $v$ from $\phi(v)$ to $\psi(v)$ at the beginning without any issues. So, we assume that $G$ has no isolated vertices. 
  
  Let $G'$ be the graph obtained from $G$ by adding
  all possible loops. We claim that the instance
  $(G',\phi,\psi)$ of $\recol(H)$ is a YES  instance if and only if $(G,\phi,\psi)$ is a YES instance.  In one direction, it is clear that if $(G',\phi,\psi)$ is a YES instance, then so is
  $(G,\phi,\psi)$. 
  
  Now, suppose that $(G,\phi,\psi)$ is a YES instance. Then there is a path from $\phi$ to $\psi$ in $\bHom(G,H)$ in which consecutive $H$-colourings differ on one vertex. It suffices to consider the case that $\phi$ and $\psi$ differ on one vertex $x$, from which the general case follows. If $\phi\psi$ is an edge of $\bHom(G',H)$, then we are done; so, we assume that it is not. Since $\phi$ and $\psi$ only differ on $x$, the only way that $\phi\psi$ can be an edge of $\bHom(G,H)$ and not $\bHom(G',H)$ is if $x$ had no loop in $G$ and $\phi(x) \nsim \psi(x)$. Observe that $\phi(x)\phi(y)\psi(x)$ is walk in $H$ for any
  neighbour $y$ of $x$ in $G'$. If there are neighbours $y$ and $z$ of $x$ such that $\phi(y) \neq \phi(z)$, then $\phi(x)\phi(y)\psi(x)\phi(z)$ forms a $4$-cycle in $H$. Therefore, $\phi$ maps all neighbours of $x$ to
  the same vertex of $H$. Let $y$ be any neighbour of $x$ in $G$ (which exists, as $x$ is not isolated). Let $\phi'$ be the map obtained from $\phi$ by changing the colour of $x$ to $\phi(y)$. Then we have $\phi \sim \phi' \sim \psi$ in $\bHom(G',H)$. Thus, there is a recolouring between $\phi$ and $\psi$.
\end{proof}

 \section{Outline of proof}\label{sect:outline} 

We assume, throughout the rest of the paper, that $H$ is a triangle-free reflexive graph and that $G$ is a reflexive graph. We also assume that $G$ and $H$ are connected; otherwise, we could simply deal with each component separately. Let $\phi$ and $\psi$ be two $H$-colourings of $G$. Our goal is to present an algorithm to determine, in polynomial time, whether there is a recolouring between $\phi$ and $\psi$ or, equivalently, whether there is a $(\phi,\psi)$-reconfiguration---a walk from $\phi$ to $\psi$ in $\bHom(G,H)$.

The main obstructions to the existence of a $(\phi,\psi)$-reconfiguration in $\bHom(G,H)$ are related to the ways in which $\phi$ and $\psi$ map the cycles (or, more generally, closed walks) of $G$ into $H$. Of course, for any cycle $C$ of $G$, a $(\phi,\psi)$-reconfiguration in $\bHom(G,H)$ yields a $(\phi_C,\psi_C)$-reconfiguration $W_C$ in $\bHom(C,H)$ between the restrictions $\phi_C$ and $\psi_C$ of $\phi$ and $\psi$ to $C$. If the length of $C$ is equal to $mk$ for some $m\geq1$ and $k\geq4$ and $\phi$ maps $C$ exactly $m$ times around an induced cycle of length $k$ in $H$, then, since $H$ is triangle-free, $\phi_C$ is an isolated vertex of $\bHom(C,H)$; such a cycle $C$ is said to be a \emph{tight} cycle of $\phi$. All homomorphisms in the same component of $\bHom(G,H)$ as $\phi$ must agree on the vertices of every tight cycle (or, more generally, tight closed walk). If $\psi$ does not agree with $\phi$ on some such vertex, then we can conclude that the desired reconfiguration does not exist.

A more subtle obstruction comes from cycles in $G$ which ``wind around'' a given induced cycle of $H$ in a non-tight way. As a specific example, for $k\geq4$, imagine that a cycle $C$ of length $3k+100$ is mapped by $\phi$ three times around an induced cycle $C'$ of length $k$ in $H$ in such a way that the final $100$ vertices are all mapped to the same vertex as the $(3k)$th vertex of $C$; note that this is a homomorphism of $C$ to $H$ because $H$ is reflexive. In this case, the cycle $C$ is not tight under $\phi$ and so $\phi_C$ is not isolated in $\bHom(C,H)$. In particular, the ``slack'' introduced by the 100 extra vertices allows $\phi_C$ to be reconfigured to any homomorphism which wraps $C$ three times around $C'$ in the same direction as $\phi_C$ does. Moreover, unlike in the tight case, the image of $C$ is not constrained to stay within $C'$. However, since $H$ is triangle-free, the ways in which the image of a vertex $v\in V(C)$ can ``leave'' the set $V(C')$ under a walk in $\bHom(C,H)$ starting with $\phi_C$ are heavily restricted. Any homomorphism in the same component of $\bHom(C,H)$ as $\phi_C$ essentially wraps $C$ three times around the cycle $C'$ in the same direction as $\phi_C$ does, with the exception of a few ``excursions'' which leave $V(C')$ for a few steps and then return to $C'$ by doubling back along essentially the same route; see Figure~\ref{fig:excursion} for an illustration and Lemma~\ref{lem:preswind} for a more formal statement and proof of this fact. 

\begin{figure}[htbp]
\begin{center}
\begin{tikzpicture}[scale=0.9]
   \newdimen\P
   \P=0.5cm
   \newdimen\PP
   \PP=1.0cm
   \newdimen\PPplus
   \PPplus=1.2cm
   \newdimen\PPPminus
   \PPPminus=1.5cm
   \newdimen\PPP
   \PPP=1.7cm
   \newdimen\PPPplus
   \PPPplus=1.9cm
   
   \begin{scope}[]
   \draw
     (0,0) node (cprime){$C'$}
      (90:\P) node [smallvert,label={[label distance=0pt]90:{}}] (v1){}
      (162:\P) node [smallvert,label={[label distance=0pt]162:{}}] (v2){}
      (234:\P) node [smallvert,label={[label distance=0pt]234:{}}] (v3){}
      (306:\P) node [smallvert,label={[label distance=0pt]306:{}}] (v4){}
      (378:\P) node [smallvert,label={[label distance=0pt]378:{}}] (v5){}
[very thick](v1)--(v2)--(v3)--(v4)--(v5)--(v1);

\draw
      (90:\PP) node [smallvert,label={[label distance=0pt]90:{}}] (v11){}
      (147:\PP) node [smallvert,label={[label distance=0pt]162:{}}] (v20){}
      (162:\PPplus) node [smallvert,label={[label distance=0pt]162:{}}] (v21){}
      (177:\PP) node [smallvert,label={[label distance=0pt]162:{}}] (v22){}
      (234:\PPplus) node [smallvert,label={[label distance=0pt]234:{}}] (v31){}
      (253:\PP) node [smallvert,label={[label distance=0pt]234:{}}] (v32){}
      (291:\PP) node [smallvert,label={[label distance=0pt]306:{}}] (v40){}
      (306:\PPplus) node [smallvert,label={[label distance=0pt]306:{}}] (v41){}
      (321:\PP) node [smallvert,label={[label distance=0pt]306:{}}] (v42){}
      (378:\PPplus) node [smallvert,label={[label distance=0pt]378:{}}] (v51){}
(v11)--(v20)--(v21)--(v22)--(v31)--(v32)--(v40)--(v41)--(v42)--(v51)--(v11);
\draw (v20)--(v21);

\draw
      (75:\PPPminus) node [smallvert,label={[label distance=0pt]90:{}}] (v110){}
      (90:\PPP) node [smallvert,label={[label distance=0pt]90:{}}] (v111){}
      (105:\PPPminus) node [smallvert,label={[label distance=0pt]90:{}}] (v112){}
      (162:\PPPplus) node [smallvert,label={[label distance=0pt]162:{}}] (v210){}
      (177:\PPP) node [smallvert,label={[label distance=0pt]162:{}}] (v221){}
      (219:\PPP) node [smallvert,label={[label distance=0pt]234:{}}] (v310){}
      (234:\PPPplus) node [smallvert,label={[label distance=0pt]234:{}}] (v311){}
      (249:\PPP) node [smallvert,label={[label distance=0pt]234:{}}] (v312){}
      (291:\PPP) node [smallvert,label={[label distance=0pt]306:{}}] (v401){}
      (306:\PPPplus) node [smallvert,label={[label distance=0pt]306:{}}] (v411){}
      (321:\PPP) node [smallvert,label={[label distance=0pt]306:{}}] (v421){}
      (363:\PPP) node [smallvert,label={[label distance=0pt]378:{}}] (v510){}
      (378:\PPPplus) node [smallvert,label={[label distance=0pt]378:{}}] (v511){}
      (393:\PPP) node [smallvert,label={[label distance=0pt]378:{}}] (v512){}
      (v110)--(v111)--(v112)--(v210)--(v221)--(v310)--(v311)--(v312)--(v401)--(v411)--(v421)--(v510)--(v511)--(v512)--(v110);
      
      \draw(v310)--(v311)--(v312);

\draw
(v1)--(v11)
(v2)--(v20)
(v2)--(v22)
(v3)--(v32)
(v4)--(v40)
(v4)--(v42)
(v5)--(v51)
(v11)--(v110)
(v11)--(v112)
(v21)--(v210)
(v22)--(v221)
(v31)--(v310)
(v32)--(v312)
(v40)--(v401)
(v41)--(v411)
(v51)--(v510)
(v41)--(v42)
(v51)--(v512);
\draw
(v42)--(v421);
\end{scope}

   \begin{scope}[xshift=5cm]
      \draw
     (0,0) node (cprime){$C'$}
      (90:\P) node [smallvert,label={[label distance=0pt]90:{}}] (v1){}
      (162:\P) node [smallvert,label={[label distance=0pt]162:{}}] (v2){}
      (234:\P) node [smallvert,label={[label distance=0pt]234:{}}] (v3){}
      (306:\P) node [smallvert,label={[label distance=0pt]306:{}}] (v4){}
      (378:\P) node [smallvert,label={[label distance=0pt]378:{}}] (v5){}
[very thick](v1)--(v2)--(v3)--(v4)--(v5)--(v1);

\draw
      (90:\PP) node [smallvert,label={[label distance=0pt]90:{}}] (v11){}
      (147:\PP) node [smallvert,label={[label distance=0pt]162:{}}] (v20){}
      (162:\PPplus) node [smallvert,label={[label distance=0pt]162:{}}] (v21){}
      (177:\PP) node [smallvert,label={[label distance=0pt]162:{}}] (v22){}
      (234:\PPplus) node [smallvert,label={[label distance=0pt]234:{}}] (v31){}
      (253:\PP) node [smallvert,label={[label distance=0pt]234:{}}] (v32){}
      (291:\PP) node [smallvert,label={[label distance=0pt]306:{}}] (v40){}
      (306:\PPplus) node [smallvert,label={[label distance=0pt]306:{}}] (v41){}
      (321:\PP) node [smallvert,label={[label distance=0pt]306:{}}] (v42){}
      (378:\PPplus) node [smallvert,label={[label distance=0pt]378:{}}] (v51){}
(v11)--(v20)--(v21)--(v22)--(v31)--(v32)--(v40)--(v41)--(v42)--(v51)--(v11);
\draw [very thick](v20)--(v21);

\draw
      (75:\PPPminus) node [smallvert,label={[label distance=0pt]90:{}}] (v110){}
      (90:\PPP) node [smallvert,label={[label distance=0pt]90:{}}] (v111){}
      (105:\PPPminus) node [smallvert,label={[label distance=0pt]90:{}}] (v112){}
      (162:\PPPplus) node [smallvert,label={[label distance=0pt]162:{}}] (v210){}
      (177:\PPP) node [smallvert,label={[label distance=0pt]162:{}}] (v221){}
      (219:\PPP) node [smallvert,label={[label distance=0pt]234:{}}] (v310){}
      (234:\PPPplus) node [smallvert,label={[label distance=0pt]234:{}}] (v311){}
      (249:\PPP) node [smallvert,label={[label distance=0pt]234:{}}] (v312){}
      (291:\PPP) node [smallvert,label={[label distance=0pt]306:{}}] (v401){}
      (306:\PPPplus) node [smallvert,label={[label distance=0pt]306:{}}] (v411){}
      (321:\PPP) node [smallvert,label={[label distance=0pt]306:{}}] (v421){}
      (363:\PPP) node [smallvert,label={[label distance=0pt]378:{}}] (v510){}
      (378:\PPPplus) node [smallvert,label={[label distance=0pt]378:{}}] (v511){}
      (393:\PPP) node [smallvert,label={[label distance=0pt]378:{}}] (v512){}
      (v110)--(v111)--(v112)--(v210)--(v221)--(v310)--(v311)--(v312)--(v401)--(v411)--(v421)--(v510)--(v511)--(v512)--(v110);
      
      \draw[very thick](v310)--(v311)--(v312);

\draw
[very thick](v1)--(v11)
(v2)--(v20)
(v2)--(v22)
(v3)--(v32)
(v4)--(v40)
(v4)--(v42)
(v5)--(v51)
(v11)--(v110)
(v11)--(v112)
(v21)--(v210)
(v22)--(v221)
(v31)--(v310)
(v32)--(v312)
(v40)--(v401)
(v41)--(v411)
(v51)--(v510)
(v41)--(v42)
(v51)--(v512);
\draw
(v42)--(v421);
\end{scope}
\end{tikzpicture}
\end{center}
\caption{Two homomorphisms from a long cycle $C$ to a graph $H$. The bold edges are the non-loop edges in the image of each homomorphism. The first homomorphism ``wraps around'' a cycle $C'$ of $H$ a bounded number of times. If $C$ is sufficiently long, then the first homomorphism can be reconfigured to the second. Triangle-freeness of $H$ makes it impossible for the image of $C$ to completely leave cycle $C'$.}
\label{fig:excursion}
\end{figure}
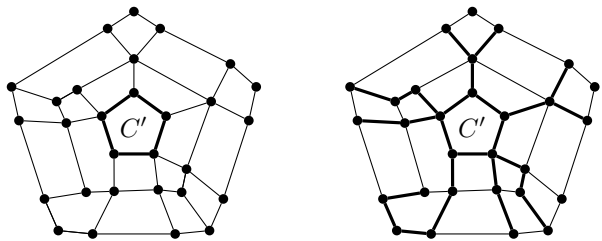


We refer to the latter obstruction as a topological obstruction.  To properly describe such obstructions we define an analogue, for triangle-free graphs, of homotopy theory and the fundamental group of a space.  The statement that the images $\phi(C)$ and $\psi(C)$ of a cycle $C$ ``wrap around the same cycles of $H$ the same number of times'' essentially translates to the closed walks that they trace out being homotopic in the fundamental group of $H$.  
The goal is to show that, if no frozen vertex or topological obstructions exist, then there is a $(\phi,\psi)$-reconfiguration.  
As it turns out, dealing with obstructions for cycles only is not enough.
It is easy to come up with examples in which there are no tight cycles, and for all cycles $C$ of $G$ the closed walks $\phi(C)$ and $\psi(C)$ are homotopic,  but there is no $(\phi,\psi)$-reconfiguration, even when $H$ is triangle-free; see Figure~\ref{fig:consistent}.   Indeed, there is a stronger topological obstruction.  If two cycles $C$ and $C'$ of $G$ share a vertex $r$, then a $(\phi,\psi)$-reconfiguration $W$ induces cycle reconfigurations $W_C$ and $W_{C'}$ that agree on $r$.  That is; where the {\em trace}
of a vertex $r$ under a reconfiguration $W = \phi_1 \sim \dots \sim \phi_d$ in $\bHom(G,H)$ is the walk 
$W_r = \phi_1(r) \sim \dots \sim \phi_d(r)$, $W_C$ and $W_{C'}$ have the same trace $W_r$. For a $(\phi,\psi)$-reconfiguration to exist, we need to be able to reconfigure the homomorphisms of all cycles of $G$ simultaneously in a consistent manner. 



\begin{figure}[htbp]
\begin{center}
\begin{tikzpicture}[scale=0.9]
   \newdimen\P
   \P=1cm
   \newdimen\PP
   \PP=1.0cm
   \newdimen\PPplus
   \PPplus=1.2cm
   \newdimen\PPPminus
   \PPPminus=1.5cm
   \newdimen\PPP
   \PPP=1.7cm
   \newdimen\PPPplus
   \PPPplus=1.9cm
   
   \begin{scope}[]
   \draw
      (0:\P) node [smallvert,label={[label distance=0pt]90:{}}] (v1){}
      (45:\P) node [smallvert,label={[label distance=0pt]162:{}}] (v2){}
      (90:\P) node [smallvert,label={[label distance=0pt]234:{}}] (v3){}
      (135:\P) node [smallvert,label={[label distance=0pt]306:{}}] (v4){}
      (180:\P) node [smallvert,label={[label distance=0pt]378:{}}] (v5){}
      (225:\P) node [smallvert,label={[label distance=0pt]378:{}}] (v6){}
      (270:\P) node [smallvert,label={[label distance=0pt]378:{}}] (v7){}
      (315:\P) node [smallvert,label={[label distance=0pt]378:{}}] (v8){}
      (\P,-1.5cm) node (G){$G$}
(v1)--(v2)--(v3)--(v4)--(v5)--(v6)--(v7)--(v8)--(v1);

\end{scope}
   \begin{scope}[xshift=2cm]
   \draw
      (0:\P) node [smallvert,label={[label distance=0pt]90:{}}] (v1){}
      (45:\P) node [smallvert,label={[label distance=0pt]162:{}}] (v2){}
      (90:\P) node [smallvert,label={[label distance=0pt]234:{}}] (v3){}
      (135:\P) node [smallvert,label={[label distance=0pt]306:{}}] (v4){}
      (180:\P) node [smallvert,label={[label distance=0pt]378:{}}] (v5){}
      (225:\P) node [smallvert,label={[label distance=0pt]378:{}}] (v6){}
      (270:\P) node [smallvert,label={[label distance=0pt]378:{}}] (v7){}
      (315:\P) node [smallvert,label={[label distance=0pt]378:{}}] (v8){}
(v1)--(v2)--(v3)--(v4)--(v5)--(v6)--(v7)--(v8)--(v1);
\end{scope}

   \begin{scope}[xshift=5cm,yshift=1cm]
   
\draw 
    (0,0) node [smallvert,label={[label distance=0pt]90:{}}] (v1){}
    (0.5cm,-1cm) node [smallvert,label={[label distance=0pt]90:{}}] (v4){}
    (-0.5cm,-1cm) node [smallvert,label={[label distance=0pt]90:{}}] (v2){}
    (0,-2cm) node [smallvert,label={[label distance=0pt]90:{}}] (v3){}
      (1cm,-2.5cm) node (H){$H$}
    (v1)--(v2)--(v3)--(v4)--(v1)
    
        (2cm,0) node [smallvert,label={[label distance=0pt]90:{}}] (u1){}
    (2.5cm,-1cm) node [smallvert,label={[label distance=0pt]90:{}}] (u4){}
    (1.5cm,-1cm) node [smallvert,label={[label distance=0pt]90:{}}] (u2){}
    (2cm,-2cm) node [smallvert,label={[label distance=0pt]90:{}}] (u3){}
    (u1)--(u2)--(u3)--(u4)--(u1)
    
            (1cm,0) node [smallvert,label={[label distance=0pt]90:{}}] (w1){}
    (1cm,-2cm) node [smallvert,label={[label distance=0pt]90:{}}] (w2){}
    (v1)--(w1)--(u1)
    (v3)--(w2)--(u3);
    
\end{scope}

   \begin{scope}[xshift=0cm,yshift=-2cm]
   
\draw 
    (0,0) node [smallvert,label={[label distance=0pt]90:{}}] (v1){}
    (0.5cm,-1cm) node [smallvert,label={[label distance=0pt]90:{}}] (v4){}
    (-0.5cm,-1cm) node [smallvert,label={[label distance=0pt]90:{}}] (v2){}
    (0,-2cm) node [smallvert,label={[label distance=0pt]90:{}}] (v3){}
      (1cm,-2.5cm) node (H){$\phi$}
    
        (2cm,0) node [smallvert,label={[label distance=0pt]90:{}}] (u1){}
    (2.5cm,-1cm) node [smallvert,label={[label distance=0pt]90:{}}] (u4){}
    (1.5cm,-1cm) node [smallvert,label={[label distance=0pt]90:{}}] (u2){}
    (2cm,-2cm) node [smallvert,label={[label distance=0pt]90:{}}] (u3){}
    
            (1cm,0) node [smallvert,label={[label distance=0pt]90:{}}] (w1){}
    (1cm,-2cm) node [smallvert,label={[label distance=0pt]90:{}}] (w2){}
    (v3)--(w2)--(u3);
    
\draw[very thick]
    (v1)--(v2)--(v3)--(v4)--(v1)
    (u1)--(u2)--(u3)--(u4)--(u1)
    (v1)--(w1)--(u1);
    
\end{scope}

   \begin{scope}[xshift=5cm,yshift=-2cm]
   
\draw 
    (0,0) node [smallvert,label={[label distance=0pt]90:{}}] (v1){}
    (0.5cm,-1cm) node [smallvert,label={[label distance=0pt]90:{}}] (v4){}
    (-0.5cm,-1cm) node [smallvert,label={[label distance=0pt]90:{}}] (v2){}
    (0,-2cm) node [smallvert,label={[label distance=0pt]90:{}}] (v3){}
      (1cm,-2.5cm) node (H){$\psi$}
    
        (2cm,0) node [smallvert,label={[label distance=0pt]90:{}}] (u1){}
    (2.5cm,-1cm) node [smallvert,label={[label distance=0pt]90:{}}] (u4){}
    (1.5cm,-1cm) node [smallvert,label={[label distance=0pt]90:{}}] (u2){}
    (2cm,-2cm) node [smallvert,label={[label distance=0pt]90:{}}] (u3){}
    
            (1cm,0) node [smallvert,label={[label distance=0pt]90:{}}] (w1){}
    (1cm,-2cm) node [smallvert,label={[label distance=0pt]90:{}}] (w2){}
    (v1)--(w1)--(u1);
    
\draw[very thick]
    (v1)--(v2)--(v3)--(v4)--(v1)
    (u1)--(u2)--(u3)--(u4)--(u1)
    (v3)--(w2)--(u3);
    
\end{scope}

\end{tikzpicture}
\end{center}
\caption{Two homomorphisms $\phi$ and $\psi$ from a graph $G$ to a triangle-free graph $H$. The bold edges indicate the non-loop edges in the images of the homomorphisms. There are no tight closed walks and the restrictions of these homomorphisms to any given cycle of $G$ can be reconfigured, but $\phi$ cannot be reconfigured to $\psi$.}
\label{fig:consistent}
\end{figure}
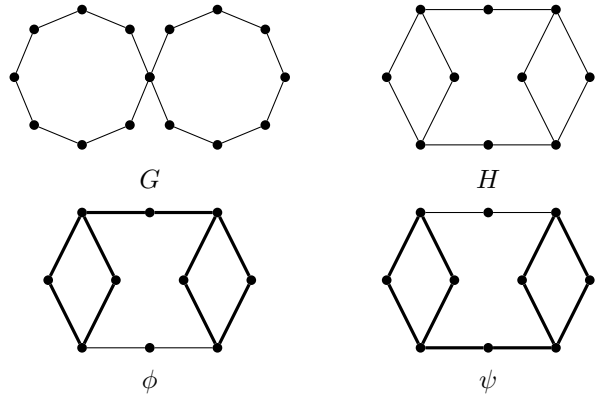

To deal with this we define, in Section \ref{sect:top}, an analogue $\pi(H,r)$ of the fundamental group of a topological space. The definition is similar to, but simpler than,  a more sophisticated definition of homotopy theory for digraphs found in \cite{LT04}.  We start with a graph $\Pi(H,r)$ whose vertices are closed walks of $H$ starting and ending at the {\em basepoint} $r$, and in which two are adjacent if they are adjacent in $\bHom(C,H)$ (where $C$ is a cycle), or if you get one from the other by subdividing an edge and mapping the new vertex to the same vertex as one of its neighbours.  The elements of $\pi(H,r)$ are the components of $\Pi(H,r)$.  The group operation is concatenation: from elements $[C]$ and $[C']$ of $\phi(H,r)$, the element $[C] \con [C']$ is the class of the closed walk $C \con C'$ with basepoint $r$ that you get by traversing $C$ and then $C'$.  We show that when  $H$ is triangle-free, classes of $\pi(H,r)$ have a canonical reduced form that can be computed in polynomial time.  

Any cycle $C$ of $H$ with a basepoint $r'$ different from $r$ can be viewed, via a `basepoint change' by a any $(r,r')$-walk $W$, as a cycle $W \con C \con \rev{W}$ in $\pi(H,r)$, where $\rev{W}$ is $W$ traversed in reverse. The trace $W_r$ of $r$ under a $(\phi,\psi)$-reconfiguration in $\bHom(G,H)$ yields a basepoint change that induces equality $[\phi(C)] = [W_r \con \psi(C) \con \rev{W_r}]$ for all cycles $C$ in $\Pi(H,r)$ at the same time. 
Following Wrochna, in \cite{Wroch20}, a walk $W_r$ in $H$ that induces this common basepoint change of all cycles is called  'topologically valid'.   Given a topologically valid walk, the only remaining obstructions are tight closed walks. 
In Section  \ref{sect:TV}, 
following the approach refined by Wrochna, give several useful equivalent conditions for the existence of a walk that is topologically valid for $\phi$ and $\psi$.  This allows us to give a polynomial time algorithm to determine if a given $(\phi(r),\psi(r))$-walk is topologically valid.  

Determining whether a given $(\phi(r),\psi(r))$-walk is topologically valid, and determining whether there exists a topologically valid $(\phi(r),\psi(r))$-walk are different problems.  For the later problem, we consider a basepoint-free version of homotopy theory in Section \ref{sect:bfhom}.  Using this basepoint-free homotopy theory we give a polynomial time algorithm to determine whether or not there is a topologically valid walk $W_r$ for given $\phi$ and $\psi$.  This is the main technical part of the paper.


In Section \ref{sect:tight} we define tight closed walks, and show that given a topologically valid walk $W_r$ one can determine in polynomial time if there are any tight closed walks that obstruct a $(\phi,\psi)$-reconfiguration that uses $W_r$.  

All these ideas are brought together in Section \ref{sect:proof} to give the formal proof of Theorem \ref{thm:mainT}. This is more descriptive version of Theorem \ref{thm:main} and yields it as an immediate corollary.

  

 
 \section{Discrete analogue of the fundamental group}\label{sect:top}

For this section, and all following sections, $H$ is a triangle-free reflexive graph.   All algorithms are about a given instance $(G, \phi, \psi)$ of $\recon(H)$, and running times are given in terms of $|V(G)|$.

 \subsection{The fundamental group of a reflexive graph}

 For any integer $\l \geq 1$ let $P_\l$ be the reflexive path on $\l+1$ vertices. 
 We refer to elements of $\bHom(P_\l,H)$ as {\em walks} of $H$, and represent them as
 $(x_0 x_1 \dots x_\l)$. If $x_0 = a$ and $x_\l = b$ the walk is an {\em  $(a,b)$-walk}.
 
 \begin{definition}\label{def:moves0}
 Let $\Pi(H;a,b)$ be the graph whose vertices are $(a,b)$-walks in $H$, and in which a walk $Y$ is adjacent to a walk $X = (x_0, \dots, x_\ell)$ if either of the following are true:
  \begin{enumerate}
  \item[\mylabel{P1}{(P1)}]
    $Y = (x_0,x_1, \dots, x_{i-1},x_i,x_i,x_{i+1},\dots, x_\l)$ for any $i$, or
  \item[\mylabel{P2t}{(P2)}]
    $Y = (x_0,x_1, \dots, x_{i-1},x'_i,x_{i+1}, \dots, x_\l)$ for some $i \notin \{0,\ell\}$, where $x_i' \sim x_i$.
  \end{enumerate}
   Of course, $Y$ is also adjacent to $X$ if $X$ is adjacent to $Y$ via this definition.   For an $(a,b)$-walk $X$, let $[X]$ be the component of $\Pi(H; a,b)$ containing $X$.    
   \end{definition}
 
  \begin{remark}
    As suggested in the previous section, we will build from this a version of $\bHom(C,H)$ in which edges of $C$ can be subdivided. 
    Property \ref{P1} allows  subdivision.  Property \ref{P2t} mimics adjacency in $\bHom(C,H)$; here we use the fact mentioned in the introduction that even in $\bHom(C,H)$ one need only change one colour at a time. 
  \end{remark}

   The {\em concatenation operation}  builds a $(a,c)$-walk $X_1 \con X_2$ from an $(a,b)$-walk $X_1$ and a $(b,c)$-walk $X_2$ by identifying the last vertex of $X_1$ with the first vertex of $X_2$.  The {\em reversal operation}  changes an $(a,b)$-walk $X_1$ to a $(b,a)$-walk $\rev{X_1}$ by reversing the order of the vertices.  
   Though concatentation is only a partial operation on the set of walks in $H$,  it is easy to check that both operations pass to $\Pi(H;a,b)$ in the following sense: for $X_1 \in  \Pi(H;a,b)$ and $X_2 \in \Pi(H;b,c)$, if   $Y_1 \in \ec{X_1}$ and $Y_2 \in \ec{X_2}$ then
   $\rev{Y_1} \in \ec{\rev{X_1}}$ and $Y_1 \con Y_2 \in \ec{X_1 \con X_2}$; so we can apply them to components of $\Pi(H; a,b)$.  Moreover $\ec{X_1 \con \rev{X_1}}$ contains the empty $(a,a)$ walk $0_a$ (or just $0$ when $a$ is understood) at $a$, which acts as an identity with respect to concatenation. Thus restricting to closed walks with a fixed first vertex, or {\em basepoint}, $r$,  it is easy to see then that concatenation is a group operation on the components of $\Pi(H;r,r)$.    We thus define the following `fundamental group' of a graph $H$.

\begin{definition}\label{def:fundgroup}
   The {\em fundamental group} of a reflexive  graph $H$, with basepoint $r$, is the set $\pi(H;r)$ of components of $\Pi(H;r,r)$ under the concatenation operation.
 \end{definition}

\subsection{The fundamental group of a triangle-free graph}

  The above construction and definition is agrees with a more general construction in  \cite{LT04}, where the authors go on to show that $\pi(H;a)$ is isomorphic to the fundamental group of the clique complex of $H$.  Without formalizing this, we simply view $\pi(H,r)$ as an analogue of the homotopy group of a space. In light of this analogy, if $\ec{C} = \ec{C'}$ we say that $C$ and $C'$ are {\em homotopic}, and if $[C] = 0$ we call it {\em contractible}.   Under this analogy, a triangle-free graph $H$ corresponds to a purley one dimensional space---a $1$-manifold, so several things become simpler.  In fact, the simplification begins with the observation that in property \ref{P2t} of Defintion \ref{def:moves0}, such $x'_i$ can exist in a triangle-free graph $H$ only if $x_{i-1} = x_{i+1}$ and this is either $x_i$ or $x'_i$.  It is not hard to see then that  the following version of Definition \ref{def:moves0} yields an $\Pi(H;a,b)$ with the same components, and so yields the same $\pi(H;r)$.       

\begin{definition}\label{def:moves1}
 Let $\Pi(H;a,b)$ be the graph whose vertices are $(a,b)$-walks in $H$, and in which a walk $Y$ is adjacent to a walk $X = (x_0, \dots, x_\ell)$ if either of the following are true:
  \begin{enumerate}
  \item[{(P1)}]
    $Y = (x_0,x_1, \dots, x_{i-1},x_i,x_i,x_{i+1},\dots, x_\l)$ for any $i$, or
  \item[\mylabel{P2}{(P2')}]
    $Y = (x_0,x_1, \dots, x_{i-1},x'_i,x_{i+1}, \dots, x_\l)$ for some $i \notin \{0,\ell\}$, where $x_{i-1}=x'_i = x_{i+1}$.
  \end{enumerate}
  For an $(a,b)$-walk $X$, let $[X]$ be the component of $\Pi(H; a,b)$ containing $X$.    
  \end{definition}

 A closed walk with a basepoint other then $r$ can be viewed as a closed walk with basepoint $r$ via what is known as a basepoint change.   For an $(r,b)$-walk $X$ and a closed walk $C$ in based at $b$ define the closed walk
 $\beta_X(C) := X \con C \con \rev{X}$. It is standard, and easily shown, that the map $\pi(H;b) \to \pi(H;r): \ec{C} \mapsto \ec{\beta_X(C)}$ is a group isomorphism.
 We observed that if $C$ is a closed walk with basepoint $r$, then a $(\phi,\psi)$-path $W$ in $\bHom(C,H)$ induces a $(\phi(r),\psi(r))$-walk $W_r$ in $\bHom(r,H)$.  In fact, it induces a homotopy between $\phi(C)$ and the cycle $\psi(C)$ with basepoint changed
 by  the trace $W_r$ of $r$.  This gives us our basic topological obstruction to a $(\phi,\psi)$-reconfiguration.  The following would hold for non triangle-free $H$ as well, but the proof would be a bit longer.

  \begin{lemma}\label{lem:preswind}
     For triangle-free $H$, if there is a $(\phi,\psi)$-path $W$ in $\bHom(C,H)$ for some cycle $C$ with basepoint $r$, then $\ec{\psi(C)} = \ec{\beta_{W_r}(\phi(C))}$ in $\pi(H;\phi(r))$.
  \end{lemma}
  \begin{proof}
    Recall that if there is a path from $\phi$ to $\psi$ in $\bHom(C,H)$, then there is a path in which any two consecutive elements differ on one vertex. By induction on the length of this path, it suffices to consider the case  that $\phi$ and $\psi$ differ on only one vertex.

    If they differ on any vertex but their basepoint, then there is nothing to prove, as then $\ec{\psi(C)} = \ec{\phi(C)}$ by \ref{P2}, and so we are done with empty $W_r$.  So we may assume that they differ on the first vertex. That is, we have
    \[ \psi(C) = (x_0, x_1, \dots x_{\l}, x_0) \mbox{ and }
       \phi(C) = (x'_0, x_1, \dots x_{\l}, x'_0).     \]
    In this case take $W_r = (x_0,x'_0)$, we will show that  
    \[ \ec{\psi(C)} = \ec{(x_0,x_1, \dots, x_{\l},x_0)}  = \ec{(x_0,x'_0, x_1, \dots, x_{\l}, x'_0,x_0)} = \ec{\beta_{W_r}(\phi(C))}.   \]
    
    Well, $x_0, x'_0$ are adjacent (as $C$ is reflexive) and are both adjacent to $x_1$ and $x_{\ell-1}$. So, as $H$ is triangle-free, we have that $x_0' = x_1$ or $x_0 = x_1$.  In the first case we have 
     \[ [(x_0,x_1, \dots, x_{\l},x_0)] =  [(x_0,x'_0 x_1, \dots, x_{\l}, x_0)] \]
    by \ref{P1} and in the second case we have it by \ref{P2}.
    Similarly $x'_0 = x_\ell$ or $x_0 = x_\ell$, and we get the second equality that we need to finish the lemma by transitivity.
  \end{proof}

\subsection{Computation of reduced cycles}

 We say that a walk $X$ in $\Pi(H;r)$ is {\em reduced} if it is a shortest walk in $\ec{X}$.
 
 \begin{lemma}\label{lem:redunique}
  Let $H$ be a triangle-free graph. Any class $\ec{X}$ in $\pi(H;r)$ has a unique reduced walk, and it can be found in linear time. 
 \end{lemma}
 \begin{proof}
   We prove the result more generally for walks in $\Pi(H;a,b)$ using the topological notion of a covering space---the idea is to define a graph $\cU$ for which the result is trivial, and to show that walks in $H$ `lift' uniqely to walks in $\cU$ in such a way that components of $\Pi$ are preserved.    

 For a graph $H$ and a vertex $r$, the {\em universal cover} $\cU$ is the infinite graph whose vertices are walks in $H$ starting at $r$, and having no occurrences of the same vertex at distance $1$ or $2$.  Two walks $X$ and $Y$ are adjacent in $\cU$ if we can get one from the other by dropping the last vertex. See Figure~\ref{fig:UC} for an example.
   
   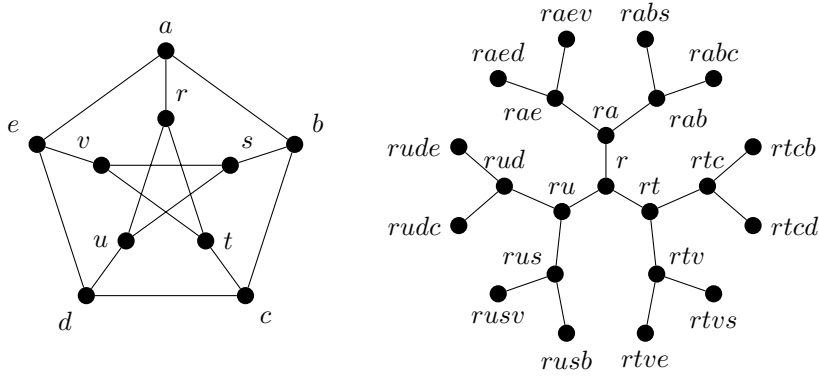
\begin{figure}[htbp]
\begin{center}
\begin{tikzpicture}[scale=0.9]
   \newdimen\P
   \P=2cm
   \newdimen\smallP
   \smallP=1cm
   \newdimen\Lone
   \Lone=0.75cm
   \newdimen\Ltwo
   \Ltwo=1.5cm
   \newdimen\Lthree
   \Lthree=2.25cm

      \draw
      (90:\P) node [vert,label={[label distance=0pt]90:{$a$}}] (v1){}
      (162:\P) node [vert,label={[label distance=0pt]162:{$e$}}] (v2){}
      (234:\P) node [vert,label={[label distance=0pt]234:{$d$}}] (v3){}
      (306:\P) node [vert,label={[label distance=0pt]306:{$c$}}] (v4){}
      (378:\P) node [vert,label={[label distance=0pt]378:{$b$}}] (v5){}
(v1)--(v2)--(v3)--(v4)--(v5)--(v1);

      \draw
      (90:\smallP) node [vert,label={[label distance=0pt]85:{$r$}}] (u1){}
      (234:\smallP) node [vert,label={[label distance=0pt]180:{$u$}}] (u2){}
      (378:\smallP) node [vert,label={[label distance=0pt]80:{$s$}}] (u3){}
      (522:\smallP) node [vert,label={[label distance=0pt]100:{$v$}}] (u4){}
      (666:\smallP) node [vert,label={[label distance=0pt]360:{$t$}}] (u5){}
(u1)--(u2)--(u3)--(u4)--(u5)--(u1);

\draw(u1)--(v1);
\draw(u2)--(v3);
\draw(u3)--(v5);
\draw(u4)--(v2);
\draw(u5)--(v4);

  \begin{scope}[shift={(6.5,0)}]
  \draw(0,0) node [vert,label={[label distance=0pt]85:{$r$}}] (r){}
  (90:\Lone) node [vert,label={[label distance=0pt]90:{$ra$}}] (ra){}
  (210:\Lone) node [vert,label={[label distance=0pt]90:{$ru$}}] (ru){}
  (330:\Lone) node [vert,label={[label distance=0pt]90:{$rt$}}] (rt){}
  
  (60:\Ltwo) node [vert,label={[label distance=-2pt]350:{$rab$}}] (rab){}
  (120:\Ltwo) node [vert,label={[label distance=-2pt]190:{$rae$}}] (rae){}
  (180:\Ltwo) node [vert,label={[label distance=0pt]90:{$rud$}}] (rud){}
  (240:\Ltwo) node [vert,label={[label distance=-2pt]170:{$rus$}}] (rus){}
  (300:\Ltwo) node [vert,label={[label distance=-2pt]10:{$rtv$}}] (rtv){}
  (360:\Ltwo) node [vert,label={[label distance=0pt]90:{$rtc$}}] (rtc){}
  
  (45:\Lthree) node [vert,label={[label distance=0pt]90:{$rabc$}}] (rabc){}
  (75:\Lthree) node [vert,label={[label distance=0pt]90:{$rabs$}}] (rabs){}
  (105:\Lthree) node [vert,label={[label distance=0pt]90:{$raev$}}] (raev){}
  (135:\Lthree) node [vert,label={[label distance=0pt]90:{$raed$}}] (raed){}
  (165:\Lthree) node [vert,label={[label distance=0pt]180:{$rude$}}] (rude){}
  (195:\Lthree) node [vert,label={[label distance=0pt]180:{$rudc$}}] (rudc){}
  (225:\Lthree) node [vert,label={[label distance=0pt]270:{$rusv$}}] (rusv){}
  (255:\Lthree) node [vert,label={[label distance=0pt]270:{$rusb$}}] (rusb){}
  (285:\Lthree) node [vert,label={[label distance=0pt]270:{$rtve$}}] (rtve){}
  (315:\Lthree) node [vert,label={[label distance=0pt]270:{$rtvs$}}] (rtvs){}
  (345:\Lthree) node [vert,label={[label distance=0pt]0:{$rtcd$}}] (rtcd){}
  (15:\Lthree) node [vert,label={[label distance=0pt]0:{$rtcb$}}] (rtcb){}
 ;
 \draw
 (r)--(ra)
 (r)--(ru)
 (r)--(rt)
 
 (ra)--(rab)
 (ra)--(rae)
 (ru)--(rud)
 (ru)--(rus)
 (rt)--(rtc)
 (rt)--(rtv)
 
 (rab)--(rabs)
 (rab)--(rabc)
 (rae)--(raed)
 (rae)--(raev)
 (rud)--(rude)
 (rud)--(rudc)
 (rus)--(rusb)
 (rus)--(rusv)
 (rtc)--(rtcb)
 (rtc)--(rtcd)
 (rtv)--(rtve)
 (rtv)--(rtvs)
 ;
  \end{scope}
\end{tikzpicture}
\end{center}
\caption{The Petersen graph and a finite portion of its universal cover.}
\label{fig:UC}
\end{figure}
   
   Because $\cU$ is a tree, it is easy to see that $\Pi(\cU; (r),(r,x_1,x_2, \dots, x_\ell))$ is connected for any vertex $(r,x_1,x_2, \dots, x_\ell)$ and that its unique shortest walk is
   \[ (r)(r,x_1)(r,x_1,x_2)  \dots (r,x_1,x_2 \dots, x_\ell). \]
   Now, the map
    \[ \phi: \cU \to H: (r,x_1, \dots, x_\ell) \mapsto x_\ell \]
   is a locally bijective homomorphism; i.e., induces a bijection between closed neighbourhoods $\{x\} \cup N(x)$ of vertices in $\cU$ and vertices in $H$. 
 It follows that the map 
    \[ \Phi: (r,x_1, \dots, x_\ell) \mapsto
        (\phi(r),\phi(r,x_1), \dots, \phi(r,x_1, \dots, x_\ell)) \] 
    is a bijective homomorphism  between the set $\Pi(\cU; (r))$ of walks in $\cU$ starting at $(r)$ and the set $\Pi(H; r)$ of walks in $H$ starting at $r$. 
    Indeed the {\em lift} $\Phi^{-1}(X)$ of a walk $X = (rx_1 \dots x_n)$ of $H$ from $\Pi(H;r)$ to $\Pi(\cU;(r))$ can be given explicitly: its first vertex is $\Phi^{-1}(X)_0 := (r)$ and its $i^{th}$ vertex is  \[ \Phi^{-1}(X)_i :=   \Phi^{-1}(X)_{i-1} \con \phi_{x_{i-1}}^{-1}(x_i) \]     where $\phi^{-1}_x$ is the bijective restriction of $\phi^{-1}$ to the closed neighbourhood of $x$.  Observe that for $x_{i-1} \sim x_i$ we have  $\phi^{-1}_{x_{i-1}}(x_i) = \phi^{-1}_{x_i}(x_i)$. 
    
   
  Now, not all $(r,b)$-walks of $H$ must lift to walks of $\cU$ with the same endpoint, but those in a component of $\Pi(H;r,b)$ should; in fact, we show now that for triangle-free $H$, a component of $\Pi(H; r,b)$ lifts isomorphically to  $\Pi(\cU;(r),B)$
  for some walk $B$, by showing $\Phi^{-1}$ is a homomorphism.   Indeed let $X$ and $X'$ be adjacent in $\Pi(H;r,b)$.  
   If $X$ and $X'$ are related by \ref{P1}, then we may assume  $X = (r,v_1, \dots, v_n)$ while $X' = (r,v_1, \dots, v_{i-1},v_i,v_i,v_{i+1}, \dots, v_n)$.  These lift to $W \con U$ and $W \con \phi^{-1}_{v_i}(v_i) \con U$ respectively where
  $W = \Phi^{-1}(r,v_1, \dots, v_i)$ as $\phi^{-1}_{v_i}(v_i)$ is equal to the last vertex $\phi^{-1}_{v_{i-1}}(v_i)$ of $W$.  But as these vertices are equal $W\con U$ and $W \con \phi^{-1}_{v_i}(v_i) \con U$  are adjacent in $\Pi(H;(r),B)$ by \ref{P1}.  
     Similar arguments about \ref{P2} complete the proof that adjacent walks in  $H$ lift to adjacent walks in $\cU$.  Thus $\Phi^{-1}$ maps a component of $\Pi(H;r,b)$ isomorphically to $\Pi(\cU,(r), B)$ for some $B$. 
     Clearly  $\Phi$ takes the shortest walk  in $\Pi(\cU;(r),B)$ to a unique shortest wak in $\Pi(H;r,b)$, so $\Pi(H;r,b)$ has a unique shortest walk. 
 
    One can lift to $\cU$ to easily find the unique shortest walk, but the fact that it is unique means that one can simply apply \ref{P1} or \ref{P2} to greedily shorten a walk, and so find 
    its reduced form in linear time. 
    
 \end{proof}



 \subsection{Computational Tools}

  We finish this section by observing some simple computational tools that are familiar from homotopy theory which we will use several times.  
  It is easy to show, for $(a,b)$-walks $X_1$ and $X_2$,  that
  \begin{equation}\label{eq:homwalks}
   \ec{X_1} = \ec{X_2} \iff \ec{ X_1 \con \rev{X_2}} = 0. 
 \end{equation}

  Taking $I$ as some initial segment of a closed walk $C$, it is not hard to see that 
  the {\em cyclic shift} $\sigma_I(C)$ of $C$ by $|I|$ vertices is in $\ec{\beta_{\rev{I}}(C)}$,
  and so 
  \begin{equation}\label{eq:perm}
    \ec{C} = [0] \iff \ec{\sigma_I(C)} = [0].
  \end{equation}  
  Thus $\ec{C}$ is contractible if and only if all cyclic shifts of it are contractible. Breaking a cycle up as the concatenation of walks, this yield identities such as the following:  
  
   \begin{equation}\label{eq:perm2}
   \ec{ X_1 \con X_2 \con X_3 \con X_4 } = [0] \iff \ec{ X_2 \con X_3 \con X_4 \con X_1 } = [0] 
   \end{equation}
 or using \eqref{eq:homwalks}: 
 \begin{equation}\label{eq:permedge}
      \ec{\rev{X_1}} = \ec{X_2 \con X_3 \con X_4} \iff \ec{\rev{X_2}} = \ec{X_3 \con X_4 \con X_1}.
    \end{equation}
    
  There is one more simple observation which allows us to apply these notions succinctly.

\begin{definition}
  When $C_1 = U_1 \con X \con V_1$ and $C_2 = U_2 \con \rev{X} \con V_2$ are closed walks
  as in Figure \ref{fig:sum},  let $C_1 \con_X C_2  = U_1 \con V_2 \con U_2 \con V_1$.
\end{definition}

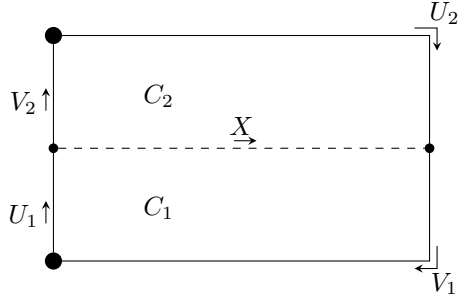
\begin{figure}
\begin{center}
\begin{tikzpicture}
   \newdimen\R
   \R=1.65cm
   \newdimen\smallR
   \smallR=1.0cm

   \draw (0,0) node [vert] (v1){};
   \draw (0,1.5) node [smallvert] (v2){};
   \draw (0,3) node [vert] (v3){};
   \draw (5,3) node [empty] (v4){};
   \draw (5,1.5) node [smallvert] (v5){};
   \draw (5,0) node [empty] (v6){};

   \draw (v1) -- (v3) -- (v4.center) -- (5,0) -- (v1); 
   \draw [dashed] (v2)--(v5); 

   \node at (-.4,.6) {$U_1$};  \draw[-stealth] (-.1, .5) -- (-.1,.8);
   \node at (-.4,2.1) {$V_2$}; \draw[-stealth] (-.1, 2) -- (-.1, 2.3);
   \node at (5.2,3.3) {$U_2$}; \draw[-stealth] (4.8,3.1) -- (5.1,3.1) -- (5.1 , 2.8);
   \node at (2.5,1.8) {$X$};   \draw[-stealth] (2.4,1.6) -- (2.7,1.6);
   \node at (5.2,-.3) {$V_1$}; \draw[-stealth] (5.1,.2) -- (5.1,-.1) -- (4.8 ,-.1);
   

  \node at (1.4,2.2) {$C_2$};
  
  \node at (1.4,.7) {$C_1$};
  
  
\end{tikzpicture}
\end{center}
\caption{The cycle $C = C_1 \con_X C_2$ where $C_1 = U_1 \con X \con V_1$
      and $C_2 = U_2 \con \rev{X} \con V_2$}
\label{fig:sum}
\end{figure}
  
 With this definition, we have the following. 
 
\begin{lemma}\label{lem:sum}
  Let $C = C_1 \con_X C_2$. If two of $\ec{C_1}, \ec{C_2}$ and $\ec{C}$ are contractible, then they all are.
\end{lemma}
\begin{proof}
  Taking shifts $C'_1 = V_1 \con U_1 \con X$ of $C_1$ and $C'_2 = \rev{X} \con V_2 \con U_1$ we have
  that $\ec{C'_1 \con C'_2} = \ec{C'}$ where $C'$ is the shift $V_1 \con U_1 \con V_2 \con U_1$
  of $C = C_1 \con_X C_2$.  It is easy to see therefore that two of $\ec{C'_1}, \ec{C'_2}$ and $\ec{C'}$ are contractible if and only if the third of them is.  The result follows by \eqref{eq:perm}.
\end{proof}

\section{Topologically valid system of walks}\label{sect:TV}

  From Lemma \ref{lem:preswind} we see a reconfiguration of $\phi$ to $\psi$ in $\bHom(G,H)$ induces for each cycle $C$ of $G$ fixed-basepoint  homotopy from $\phi(C)$ to $\beta_{W_r}(\psi(C))$ with a common basepoint change $W_r$.   
Not only is such a walk  $W_r$  necessary for a reconfiguration it is almost sufficient.  If it exists, the only obstructions a reconfiguration will be non-topological (tight closed walks).  So recognizing the existence of such a
basepoint change is key.  This task is simplified by the fact that choice of the basepoint $r$ is unimportant.    We address this in this section. 

  \begin{definition}\label{def:TV}
   A {\em system of walks for $\phi$ and $\psi$} is a vector $(W_v):= (W_v)_{v \in V(G)}$ of walks in $H$ such that $W_v$ is a $(\phi(v),\psi(v))$-walk.  A walk $W_v$ is  {\em topologically valid} if for  every  closed walk
  $C$ of $G$ with basepoint $v$ we have 
    \begin{equation}\label{eq:TV}
      \ec{\phi(C)} = \ec{\beta_{W_v}(\psi(C))}.
    \end{equation}
  The system of walks is {\em topologically valid} if all walks in it are topologically valid.  
  \end{definition}

  As suggested above, we get a topologically valid sytem of walks from a reconfiguration. 

  \begin{fact}\label{fact:necTV}
    Let $W$ be a reconfiguration from $\phi$ to $\psi$ in $\bHom(G,H)$. Where for every vertex $v$ of $G$,  $W_v$ is the trace of $v$, the system $(W_v)$ of walks is topologically valid for $\phi$ and $\psi$. 
  \end{fact}
  \begin{proof}
    In Section \ref{sect:outline} we observed that $W$ induces a $(\phi(C),\psi(C))$-walk in $\bHom(C,H)$ for any closed walk $C$ of $G$.  In Lemma \ref{lem:preswind} we observed that where $v$ is the basepoint of $C$ this implies $\ec{\psi(C)} = \ec{\beta_{W_v}(\phi(C)}$ in $\pi(H; \phi(v))$. Thus any reconfiguration $W$ from $\phi$ to $\psi$  yields a system $(W_v)$ of walks that is topologically valid for $\phi$ and $\psi$.
  \end{proof}

   Generalizing the condition for closed walks, a $uv$-walk $U$ in $G$ is {\em (topologically) preserved} by a system $(W_v)$ of walks for $\phi$ and $\psi$ if 
   \begin{equation}\label{eq:TV1}
      \ec{\phi(U)} = \ec{W_u \con \psi(U) \con \rev{W_v}}. 
    \end{equation}  
    
   By \eqref{eq:homwalks} and \eqref{eq:perm} we can write in the following form: 
     \begin{align}
    \ec{\phi(U) \con W_v \con \rev{\psi(U)} \con \rev{W_u}} = 0  &     \label{eq:TV3}\tag{\ref{eq:TV1}$'$}
  \end{align}
   
  Now, for a $uw$-walk $U$ and a $wv$-walk $V$ we have
  \begin{multline*}
    \phi(U \con V) \con W_v \con \rev{\psi(U \con V)} \con \rev{W_u}  = \\
     \left( \phi(U) \con W_w \con \rev{\psi(U)} \con \rev{W_u} \right) \con_{W_w}
    \left( \phi(V) \con W_v \con \rev{\psi(V)} \con \rev{W_w} \right),
  \end{multline*}
  so by Lemma \ref{lem:sum} and form \eqref{eq:TV3}, we get that if two of $U, V$ and $U \con V$ are preserved by $(W_v)$, then so is the third of them. This yields the following alternate definition. 

  \begin{fact}\label{fact:TV'}
    A system of walks $(W_v)$ is topologically valid for $\phi$ and $\psi$ if and
    only if every edge in $G$ is topologically preserved by $(W_v)$ if and only if 
    every walk in $G$ is toplogically preserved by  $(W_v)$.
  \end{fact}

  The topologically valid system of walks $(W_v)$ we get from a reconfiguration $W$ is not generally unique, indeed each walk in general may have many repeated vertices, but it is unique up to homotopy classes. 
     
   \begin{lemma}\label{lem:uniquesystem}
      Let $W_r$ be a reduced topoligically valid  $(\phi(r),\psi(r))$-walk in $H$.  There is a unique topologically valid system $(W_v)$ of reduced walks for $\phi$ and $\psi$ that contains $W_r$.  
   \end{lemma}
   \begin{proof}    
     First we show that there is a topologically valid system $(W_v)$ of walks containing $W_r$.  Choose a spanning tree $T$ of $G$ with root $r$.  Inductively build $(W_v)$ by setting $W_v$ to be the reduction of $\phi(vu) \con W_u \con \psi(uv)$, where the $(\phi(u),\psi(u))$-walk $W_u$ is already defined for the precursor $u$ of $v$ in $T$.
    
    By construction any edge in $T$, so any path in $T$ is preserved by $(W_v)$.  Thus the system $(W_v)$ is topologically valid if and only if every edge $e$ in $G \setminus T$ is also preserved by $(W_v)$. 

    If there is some edge $e = uv$ that is not preserved by $(W_v)$.  Letting $T_v$ denote the unique path in $T$ from $r$ to $v$, we clearly have then that the closed walk $C =  T_u \con e \con \rev{T_v}$ is not preserved by $(W_v)$, so $\ec{\phi(C)} \neq \ec{\beta_{W_r}(\psi(C))}$, contradicting the fact that $W_r$ was topologically valid.  
   \end{proof}

       In light of this fact, we make the following definition.  

    \begin{definition}
       A topologically valid $(\phi(r),\psi(r))$-walk $W_r$ contained in a topologically valid system $(W_v) = (W_v)_{v \in V(G)}$ of walks,  {\em generates} the system $(W_v)$. 
    \end{definition}

     Actually our proof of Lemma \ref{lem:uniquesystem} did more.  Defining $(W_v)$ as we did in the proof can clearly be done in polynomial time.  Once we have done so, checking that  at most quadratic number of edges not in $T$ are preserved can be done in polynomial time. 
     Thus we showed the following.

  \begin{corollary}\label{cor:TValg}
    Let $W_r$ be reduced walk from $\phi(r)$ to $\psi(r)$. We can determine in polynomial time if $W_r$ is topologically valid. If it is, we produce the system $(W_v)$ it generates, if not, we provide a closed walk $C$ of $G$ such that $\ec{\phi(C)} \neq \ec{\beta_{W_r}(\psi(C)}$.
  \end{corollary}

   \section{Basepoint-free homotopy}\label{sect:bfhom} 

   Fact \ref{fact:necTV} tells us that there is a $(\phi,\psi)$-reconfiguration in $\bHom(G,H)$  only if there is a topologically valid walk $W_r$ between the basepoints $\phi(r)$ and $\psi(r)$,
  and Corollary \ref{cor:TValg} tells us how to decide if a given $W_r$ is topologically valid. 
   Our main result in this section, Lemmma \ref{lem:algStep1}, tells us how to decide if there is any
   walk $W_r$ that is topolocially valid for $\phi$ and $\psi$.  
   If a randomly chose $(\phi(r),\psi(r))$ walk $W_r$ is not valid, then Corollary \ref{cor:TValg} gives us a non-contractible closed walk $C$ such that $\ec{\phi(C)} \neq \ec{\beta_{W_r}(\psi(C)}$. This $C$ is our starting point in this section, as it greatly limits the possibilites for a valid $W_r$.  To characterise the possibilites, we use a version of $\Pi(H,r)$ in which we let the basepoint move freely.

    
    \begin{definition}
      Let $\Pi(H)$ be the graph we get from the disjoint union, over all $r \in V(H)$, of $\Pi(H;r)$ 
    by adding an edge between walks $y$ and $x = (x_0,x_1, \dots,x_{\ell-1},x_0)$ if 
      \begin{enumerate}
          \item[\mylabel{P3}{(P3)}] $y = (x_1, \dots, x_{\l-1})$ and $x_1 = x_{\l-1}.$
      \end{enumerate}
    \end{definition}
    
     Let $\pi(H)$ be the set of components of $\Pi(H)$, and let $\cl{C}$ be the component of $\Pi(H)$ containing $C$.   Clearly we have that
     \begin{equation}\label{eq:equivhom}
        \ec{C} = \ec{C'}  \mbox{ implies }   \cl{C} = \cl{C'}. 
    \end{equation}        
    The converse is not true in general, indeed not even heuristically, as $C$ and $C'$ may satisfy $\cl{C} = \cl{C'}$ and have different basepoints.   However, if we have a walk $W$ from $C$ to $C'$ in $\Pi(H)$, then it is not hard to see that $\ec{C} = \ec{\beta_{W_r}(C')}$ where $W_r$ is the trace of the basepoint $r$ of $C$.

    If $C$ is a closed walk with basepoint $r$, then it is trivial to see that $\cl{C} = \cl{\beta_W(C)}$ for any walk $W$ ending  at $r$, and so from Lemma \ref{lem:preswind} we get that if there is a $(\phi,\psi)$-path in $\bHom(C,H)$ 
    then $\cl{\phi(C)} = \cl{\psi(C)}$. A closed walk $C$ is {\em free-reduced} if it is a shortest closed walk in $\cl{C}$.  Though any class in $\pi(H;r)$ contains a unique reduced closed walk, this does not hold for free-reduced closed walks in classes of $\pi(H)$. Indeed, if $C$ is free-reduced, any cyclic shift of it is also free reduced. However, any non-contractible reduced closed walk $C$, `contains' a unique free-reduced closed walk. 
    
    \begin{fact}\label{fact:redcycle}
     If $C$ is a non-contractible reduced closed walk with basepoint $r$, then $C$ decomposes uniquely as 
     $C = \beta_T(C_f)$ for some free reduced closed walk $C_f$, and some reduced walk $T$ from $r$ to the first vertex of $C_f$.  
    \end{fact} 
    \begin{proof}
       Indeed, as $C = (x_0,x_1, \dots, x_{\l-1},x_0)$ is reduced with respect to operations \ref{P1} and \ref{P2} the only possible reduction is operation \ref{P3}, meaning $x_{\l-1} = x_1$ and it reduces to $C_1 = (x_1, x_2 \dots, x_{\l-1})$.  Where $T_1 = (x_0,x_1)$ we have $C_1 = \beta_{T_1}(C_1)$.  By induction we get the fact. 
    \end{proof} 
    
     As the reduced form of $C$ is unique and can be found in linear time, so can its decomposition
     into $\beta_T(C_f)$.  We call this its {\em free decomposition} and call $T$ its {\em tail}.  
     
  This free decomposition does not change by much under basepoint change. To show this we start with a somewhat technical calculation. The intuition for this lemma is perhaps aided by first reading the corollaries that follow it. 
  
  A walk $I$ is {\em initial} in a walk $X$ if $X = I \con J$ for some $J$.  A walk is {\em terminal} in $X$ if it is initial in $\rev{X}$. 
  
  \begin{lemma}\label{lem:fbc}
    Let $A$ and $B$ be non-contractible reduced closed walks with respective free decompositions
    $\beta_T(A_f)$ and $\beta_S(B_f)$.  If $\ec{A} = \ec{\beta_W(B)}$ for a walk $W$ 
    then there is an integer $d$ and initial $I$ of $A_f^d$ such that $B_f = \sigma_I(A_f)$ is a cyclic shift of $A_f$, and $\ec{W} = \ec{T \con I \con \rev{S}}$. This is reduced unless $T$ and $S$ end in a common walk.
    
  \end{lemma}
  \begin{proof}
    As $\ec{A} = \ec{\beta_W(B)}$ we have $\ec{A^d} = \ec{\beta_W(B^d)}$ so $\ec{A^dW} = \ec{WB^d}$ for any integer $d$. While $A^dW$ and $WB^d$ are not generally reduced, $A^d$ and $B^d$ reduce to $\beta_T(A_f^d)$ and $\beta_S(B_f^d)$, and so by taking $d$ large enough, we can say $A^dW = WB^d$ reduces to 
    \begin{equation}\label{eq:com}
         I_A \con T_W = I_W \con T_B\tag{*} 
    \end{equation}  
    for some initial $I_A$ of $\beta_T(A_f^d)$,  some termimal $T_W$ of $W$, some initial $I_W$ of $W$ and some terminal $T_B$ of $\beta_S(B_f^d)$.  As \eqref{eq:com} holds for all large enough $d$,  and $I_A$ and $T_B$ grow with $d$, we have that $|A_f| = |B_f|$, and that by taking $d$ large enough,  $I_A$ and $T_B$ can be assured to have an arbitrarily long overlap, so $B_f$ is some shift of $A_f$, as needed.  We now just have to verify the decomposition of $W$. 
    
    Again from \eqref{eq:com} for large enough $d$ we get that the end $T_W$ of $W$ is terminal in $T_B$ so in $B_f^d \con \rev{S}$, and as $\beta_T(A_f^d) \con W$ reduced to $I_A \con T_W$, we get that $W = W' \con T_W$ for some $W'$ which cancels with the end of $A_f^d \con \rev{T}$.  Thus $W'$ is initial in $T \con A_f^{-d}$. 
    
     Now, $W'$ ends where $T_W$ starts. If $T$ and $S$ end in a common walk, then it is possible that $W' \con T_W$ is the reduction of $T \con S$.  Otherwise $W'$ contains $T$ and $T_W$ contains $S$.  In this case $W'$ begins with $T$ and then contains some initial walk of $A_f^{-d}$.  So $W' \con T_W$ is $T \con I \con \rev{S}$ for some initial $I$ of $A^d_f$ or $A_f^{-d}$, as needed. That $B_f = \beta_I(A_f)$ now follows from the fact that $B_f$ is a shift of $A$, and that $I$ is an initial walk in $A_f$ or its inverse that has the right endpoints.  
  \end{proof}

  From this we get a couple of immediate corollaries.

    \begin{corollary}\label{cor:freeredcycle}
      If $A_f$ and $B_f$ are free reduced non-contractible closed walks with $\cl{A_f} = \cl{B_f}$
      then $B_f$ is a cyclic shift of $A_f$.
    \end{corollary} 
    \begin{proof}
       As $\cl{A_f} = \cl{B_f}$ we have $B_f \in [\beta_{W_r}(A_f)]$ where $W_r$ is the trace of the basepoint $r$ under the walk $W$ from $A_f$ to $B_f$ in $\Pi(H)$.  By the lemma $B_f$ is a shift of $A_f$.
    \end{proof}
  
    For a reduced closed walk $C$, let $\sqrt{C}$ be the shortest initial walk of $C$ such that $\sqrt{C}^d = C$ for some positive $d$.  
   
    \begin{corollary}\label{cor:formhomotopy}
       If $A$ and $B$ are non-contractible reduced closed walks with free decompositions $A = \beta_T(A_f)$ and $B = \beta_S(B_f)$, then $\ec{A} = \ec{\beta_W(B)}$ if and only if $W$ is of the form 
             \[ W_d = \sqrt{A}^d \con T \con I \con \rev{S} \]
        for some integer $d$, where $I$ is an initial walk of $A$. 
    \end{corollary} 
    \begin{proof}
     The `if' part is trivial as $\ec{W_d} = T \con \sqrt{A_f}^d \con I \con \rev{S}$ and this is the trace of the path from $A$ to $B$ that we get by composing the reduction of $A$ to $A_f$ with the
     cyclic shift of $A_f$ by $A_f^d \con I$ to $B_f$, and then composing this with the inverse of the reduction of $B$ to this shift of $B_f$.  The `only if' is from Lemma \ref{lem:fbc} the lemma.  
    \end{proof}


   \begin{corollary}\label{cor:uniqpath}
      Let $A$ and $B$ be reduced closed walks in $G$ with basepoint $r$ having respective free decompositions 
      $\beta_T(A_f)$ and $\beta_S(B_f)$, and such that $\sqrt{A} \neq \pm \sqrt{B}$.
      If there exists a walk $W_r$ such that $\ec{\phi(A)} = \ec{ \beta_{W_r}{\psi(A)}}$ and $\ec{\phi(B)} = \ec{ \beta_{W_r}{\psi(B)}}$,
      then there is a unique reduced such walk and it is the reduction of either $T \con I \con \rev{S}$
      for some initial $I$ of $A_f$  or $T' \con I' \con \rev{S'}$ for some initial $I'$ of $B_f$.   
   \end{corollary}
   \begin{proof}
     Let $A$ and $B$ and $W$ be as in the premise of the corollary.
     By Corollary \ref{cor:formhomotopy} we have then 
      \[ \ec{\sqrt{A}^d \con T \con I \con \rev{S}} = \ec{\sqrt{B}^{d'} \con T' \con I' \con \rev{S'}} \]
     for some $d$ and $d'$ where $I$ and $I'$ (and all other components) are fixed.  
     Reducing these, we get 
      \[ \ec{T \con \sqrt{A_f}^d \con I \con \rev{S}} = \ec{T' \con \sqrt{B}_f^{d'} \con I' \con \rev{S'}} \]
     where $A_f$ is the free reduction of $A$ and $B_f$ is the free reduction of $B$.  As $\sqrt{A} \neq \sqrt{B}$ we have either $T \neq T'$ or $\sqrt{A_f} \neq \sqrt{B}_f$.  Either way, the only possible solution has $d = 0$ or $d' = 0$.
     These cannot yield different solutions by considerations of length. 
      \end{proof}  

  We are now ready to prove the main result of the section.

  \begin{lemma}\label{lem:algStep1}
         Let $r$ be a vertex of $G$. We can determine in polynomial time if there is a $(\phi(r),\psi(r))$-walk $W_r$ that is topologically valid for 
         $\phi, \psi$ in $\bHom(G,H)$.  If there is, one with length at most $2n$ is provided.   
  \end{lemma} 
    \begin{proof}
        Let $W_r$ be any shortest $(\phi(r),\psi(r))$-walk in $H$. By Corollary \ref{cor:TValg} we can determine in polynomial time if $W_r$ is  topologically valid. If it is, we are done, so we may assume that it is not, and the algorithm returns a closed walk 
        $C$ of $G$ such that $\ec{\phi(C)} \neq \ec{\beta_{W_r}(\psi(C))}$, so necessarily $\phi(C)$ is non-contractible.
        
        By Lemma \ref{lem:redunique}, (and Fact \ref{fact:redcycle}) we can find the free decompositions  $\phi(C) = \beta_T(C_f)$ and $\psi(C)=\beta_S(C'_f)$ in polynomial time. 
        We can check in polynomial time if $C'_f$ is a shift of $C_f$. If it is not, then $\cl{\phi(C)} \neq \cl{\psi(C)}$ by Corollary \ref{cor:freeredcycle}, and so by \eqref{eq:equivhom}, we have
         \[         \ec{\phi(C)} \neq \ec{\beta_W(\psi(C))}\]
        for any $W$. Thus there can be no topologically valid $W_r$, and we are done.
          
        We may assume, therefore, that $C'_f$ is the shift $\sigma_I(C_f)$ for some initial walk $I$ of $C_f$.  By Corollary \ref{cor:formhomotopy} we thus have that the only reduced walks $W$ that preserve $C$ are reductions of 
               \[ W_d:= \sqrt{\phi(C)}^d \con T \con I \con \rev{S} \]
         for integers $d$.    
        
        By Corollary \ref{cor:TValg} we can  decide if $W_0$ is topologically valid. If it is we return it and are done, so we may assume that it is not, and Corollary \ref{cor:TValg} yields another closed walk $C'$ with $\ec{\phi(C')} \neq \ec{\beta_{W_0}(\psi(C'))}$.
        By Corollary \ref{cor:formhomotopy} we get that $\sqrt{C'} \neq \pm\sqrt{C}$ and so by Corollary \ref{cor:uniqpath}  the only walks $W$ that can preserve both $C$ and $C'$ are the reductions of $T \con I \con \rev{S}$ and $T' \con I' \con \rev{S'}$ where $\beta_{T'}(C_f)$ is the reduction of $\phi(C')$ and $\beta_{S'}(\sigma_{I'}(C_f))$ is the reduction of $\psi(C')$.   Both of these have length at most $2n$ so we can check them for topological validity in polynomial time by Corollary \ref{cor:TValg}, and return the shortest one that is valid if any are.
    \end{proof}

\section{Tight closed walks}\label{sect:tight}

 Though having a topologically valid $(\phi(r),\psi(r))$-walk in $H$ guarantees there are no topological obstructions to a $(\phi,\psi)$-reconfiguration, there may still be non-topogical obstructions.

  \begin{definition}\label{def:realizable}
    A system of walks $(W_v)$ is {\em realizable} if (by possibly adding repeated vertices into walks) we get a system of walks that is induced by a reconfiguration.   
  \end{definition}

   \begin{definition}\label{def:tight}
      A closed walk $C$ of $G$ is {\em tight} under $\phi$ if $\phi(C)$ is  free reduced in $\pi(H)$.
      Any vertex of a tight closed walk is {\em fixed} under $\phi$. 
   \end{definition}
   
   It is not hard to see that if $C$ is tight under $\phi$, then any neighbour of $\phi$ in $\bHom(G,H)$ must agree with $\phi$ on $C$. Thus a fixed vertex is one that cannot change under a reconfiguration. A walk is {\em constant} if all vertices are the same. 
   
   \begin{fact}\label{fact:nectight}
       If the system of walks $(W_v)$ is realizable for $\phi$ and $\psi$ in $\bHom(G,H)$ and $c$ is fixed under $\phi$, then $W_c$ is constant.    
   \end{fact}

  \begin{lemma}\label{lem:RealizAlg}
    For a given system $(W_v)$ of paths that is topologically valid for  $\phi, \psi \in \bHom(G,H)$, we can decide in time linear in the sum of the lengths of the paths if $(W_v)$ is realizable.  
    If it is not realizable, we give a closed walk $C$ on which $\phi$ is tight. 
  \end{lemma}  
  \begin{proof}
  Let homomorphisms $\phi, \psi :G \to H$ and a system of walks $(W_v)_{v \in G}$
  that is topologically valid for $\phi$ and $\psi$ be given.  Starting with $\phi$ we attempt to construct a reconfiguration $W$ from $\phi$ to $\psi$, by changing one vertex $v$ at a time from its current position $W_v$ to its next position. If we succeed, this witnesses the realizability of $(W_v)$.
  As any edge $e = uv$ of $G$ is topologically valid under the system of walks, 
  \[ [W_u] =  [ \phi(e) \con W_v \con \rev{\phi(e)}], \]
  and as by Lemma \ref{lem:redunique} there is a unique reduced walk in $[W_u]$, we have, where a superscript of $i$ on the walk designates that $i^{th}$ vertex of the walk, either  $W_u^i = W_v^{i-1}$ for all $i$ except maybe $i = |V(W_u)|$,  
        or      $W_v^i = W_u^{i-1}$ for all $i$ except maybe $i = |V(W_v)|$.  
  
  Construct an auxillary digraph $A$ on $V(G)$ by setting $u \to v$ if
  \begin{equation}\label{eq:aux} 
    W_v^1 = W_u^0 \mbox{ and } W_v^0 \nadj W_u^1. 
  \end{equation}    
  (Interpret this as '$u$ wants to move to $W_u^1$, but $v$ has to move first'.) 
  Observe that if $uv$ is an edge, and $|W_u| = 0$, then $|W_v| \leq 1$, and $W_v^1 \sim W_u^0$,
  so $u$ has no arcs in $A$. 
  
  If $A$ has any sink $v$, then on moving $v$ to $W_v^1$, we still have a homomorphism. Append this 
  homomorphism to $W$, and relabelling $W_v$ by removing its first vertex. As each relabelling reduces the length of a walk in $(W_v)$ by one, we either complete in time linear in the sum of the lengths of the paths, (and we are done) or at some step $A$ has no sinks, so contains a directed cycle.  

  Assume that $A$ has a directed cycle $C = a_1 \to a_2 \to \dots \to a_t$.
  By construction, every arc $a_i \to a_{i+1}$ of $C$ is an edge $a_i \sim a_{i+1}$ of $G$, and \eqref{eq:aux} ensures that $\phi(a_{i}) \notin \{ \phi(a_{i-1}), \phi(a_{i-2}) \}$,  (indices modulo $t$) for all $i$. 
  Thus is $C$ is a tight closed walk under $\phi$. 
  
\end{proof}

\section{Full statement and proof of the main theorem}\label{sect:proof}

We are now ready to state and prove the main technical theorem,  from which Theorem \ref{thm:main} follows immediately. 
   
 \begin{theorem}\label{thm:mainT}
    Let $H$ be a reflexive symmetric triangle-free graph, and $(G, \phi, \psi)$ be an instance of $\recol(H)$.  There is a reconfiguration from $\phi$ to $\psi$ if and only if
     there is a system of walks $(W_v)_{v \in V(G)}$ that is topologically valid for $\phi$ and $\psi$, such that $W_c$ is a constant walk for any vertex $c$ in closed walk $C$ that is tight under $\phi$.  
     
    Moreover, these validity of this conditions can be determined by an algorithm that runs in time that is polynomial in $|V(G)|$. 
 \end{theorem}
\begin{proof}
  The `only if' part of the first statement of the theorem is proved in Facts \ref{fact:necTV} and \ref{fact:nectight}.   The `suffciency' is proved in Lemma \ref{lem:RealizAlg}.  
  The validity algorithm comes from Lemma \ref{lem:algStep1} and Lemma \ref{lem:RealizAlg}.

 Indeed, choosing a vertex $r$ of $G$,  by Lemma \ref{lem:algStep1} we determine if there is a topologically valid $(\phi(r),\psi(r))$-path.  If there is not, we are done, and if there is, we get one with $W_r$ of length at most $2n$.  
 In fact we get the system $(W_v)$ of walks that it generates, and no path in this system can have length greater than $4n$ so the sum of the lengths of these paths are polynomial in $n$.   By Lemma \ref{lem:RealizAlg} we determine
 if this system of walks is realizable in polynomial time.  If it is we are done,  and if it is not, then the lemma returns a tight closed walk.  Choosing new $r$ in $C$, any realizable system of paths is constant on $r$. 
 If  $\phi(r) \neq \psi(r)$ there is no such system, and we are done.  If  $\phi(r) = \psi(r)$, set $W_r$ to be constant, and apply Corollary \ref{cor:TValg} to determine if  it is topolgically valid.  If it is apply Lemma \ref{lem:RealizAlg} we determine if it is  realizable.

\end{proof}

  %


\newcommand{\doi}[1]{\href{http://dx.doi.org/#1}{\small\nolinkurl{DOI: #1}}}
\renewcommand{\url}[1]{\href{https://arxiv.org/abs/#1}{\small\nolinkurl{arXiv: #1}}}
\bibliographystyle{plainnat}

\end{document}